\documentclass[11pt,a4paper]{article}
\usepackage{amsfonts}
\usepackage{amssymb}
\usepackage{mathrsfs}
\usepackage{amsmath}
\usepackage{booktabs}
\usepackage{epsf,epsfig,amsfonts,amsgen,indentfirst}
\usepackage{amsmath,amstext,amsbsy,amsopn,amsthm,bbding,wasysym}
\usepackage{multicol,mathdots}
\usepackage{subfigure}
\allowdisplaybreaks

\newtheorem{theorem}{Theorem}[section]

\newtheorem{lemma}[theorem]{Lemma}

\newtheorem{Problem}{Problem}

\begin{document}
\title
{\LARGE \textbf{Solution to an open problem on Laplacian ratio\thanks{ supported by NSFC (No. 12261071)and NSF of Qinghai Province (No. 2020-ZJ-920).} }}

\author{ Tingzeng Wu$^{a,b}$\thanks{{Corresponding author.\newline
\emph{E-mail address}: mathtzwu@163.com, a3566293588@163.com, hjlai@math.wvu.edu, zxl2748564443@163.com
}}, Xiangshuai Dong$^a$, Hong-Jian Lai$^c$, Xiaolin Zeng$^a$\\
{\small $^{a}$ School of Mathematics and Statistics, Qinghai Nationalities University, }\\
{\small  Xining, Qinghai 810007, P.R.~China} \\
{\small $^{b}$ Qinghai Institute of Applied Mathematics,   Xining, Qinghai, 810007, P.R.~China} \\
{\small $^{c}$ Department of Mathematics, West Virginia University, Morgantown, WV, USA} }
\date{}

\maketitle
\noindent {\bf Abstract:} Let $G$ be a graph.  The Laplacian ratio of $G$ is the permanent of the  Laplacian matrix of  $G$ divided   by the product of  degrees of all vertices. The computational complexity of Laplacian ratio is $\#{\rm P}$-complete. Brualdi and Goldwasser studied systematicly the  properties  of   Laplacian ratios of graphs. And they proposed an open problem: what is the minimum value of the Laplacian ratios of trees with $n$ vertices having diameter at least $k$ ?  In this paper,  we give a solution to the problem.

\section{Introduction}%
 The {\em permanent} of an~$n$~square matrix~$M=[m_{ij}]$~with~$i, j \in \{1,2,\ldots,n\}$,~is defined as
\begin{eqnarray*}
{\rm per}M=\sum_{\sigma\in \Lambda_{n}}\prod_{i=1}^{n}m_{i\sigma(i)},
\end{eqnarray*}
where $\Lambda_{n}$ denotes the symmetry group of order $n$. Valiant \cite{val} has shown that computing the permanent is $\#{\rm P}$-complete even  when restricted to {\rm(0,1)}-matrices.

Let $G=(V(G),E(G))$ be a simple  graph with $V(G)=\{v_{1},v_{2},\ldots,v_{n}\}$ and $d_{i}$ denote the degree of vertex $v_{i}$ in $V(G)$ for $i=1,2,\ldots, n$. Let $D(G)$ be the diagonal matrix whose $(i,i$)-entry is $d_{i}$, and let $A(G)$ be the adjacency matrix of $G$. The matrix $L(G)=D(G)-A(G)$ is the {\em Laplacian matrix} of $G$. Let $PD(G)$ be the product of degrees all vertices in $G$. If $G$ has a $u,v$-path, then the {\em distance} from $u$ to $v$, denoted $d(u,v)$, is the least length of a $u,v$-path. The {\em diameter} of $G$ is $\max_{u,v\in V(G)}d(u,v)$. Suppose that $G-v$ and $G-uv$ denote the graph obtained from $G$ by deleting vertex $v\in V(G)$, or edge $uv\in E(G)$, respectively (this notation is
naturally extended if more than one vertex or edge is deleted). The path and star  of order $n$ are denoted by $P_{n}$ and $S_{n}$, respectively.  A {\em caterpillar} is a tree in which a single path (the spine) is incident  to (or containing) every edge. A double-star $DS(p,q)$ is a tree obtained by joining the center of stars $S_{p}$ and $S_{q}$ by one edge, where $p\geq2$ and $q\geq 2$.

The  {\em Laplacian ratio} of $G$ is defined by
\begin{eqnarray*}
\pi(G)=\frac{{\rm per}L(G)}{PD(G)}.
 \end{eqnarray*}
 By the definition as above, we can conclude that computing the  Laplacian ratio of a graph is $\#{\rm P}$-complete because of Valiant's result \cite{val} mentioned earlier.  Laplacian ratios of graphs were first considered by Brualdi and Goldwasser \cite{bru}. They determined the bounds for $\pi(G)$ when $G$ is a tree or generally, a bipartite graph.
In particular, they proposed three open problems for {\em Laplacian ratios} of trees in 1984. That is,

\begin{Problem}\label{prob1} (\cite{bru})
What is the minimum value of the Laplacian ratios of trees with $n$ vertices having diameter at least $k$?
\end{Problem}

\begin{Problem}\label{prob2} (\cite{bru})
What is the minimum value of the Laplacian ratios of trees with $n$ vertices having a $k$-matching?
\end{Problem}

\begin{Problem}\label{prob3} (\cite{bru})
Let $\mathscr {T}_{1}(n,k)$ be the family of all trees of order $n$ with diameter at most $k$, and $\mathscr {T}_{2}(n,k)$ be the family of all trees of order $n$ with matching number at most $k$. Determine the maximum value of  Laplacian ratios of graphs in  $\mathscr {T}_{1}$ and in $\mathscr {T}_{2}$.
\end{Problem}

Over the last forty years, only Problem \ref{prob2} is solved in \cite{gol}.  Problems \ref{prob1} and \ref{prob3}  remain open. In this paper,  our interest is to investigate Problem \ref{prob1}. Let $n$ and $k$  be integers. The Broom graph  $B(n,k)$ is a graph obtained  by attaching $n-k$ pendant edges to an endvertex of $P_{k}$.  We  prove the following main result.
\begin{figure}[htbp]
\begin{center}
\includegraphics[scale=0.5]{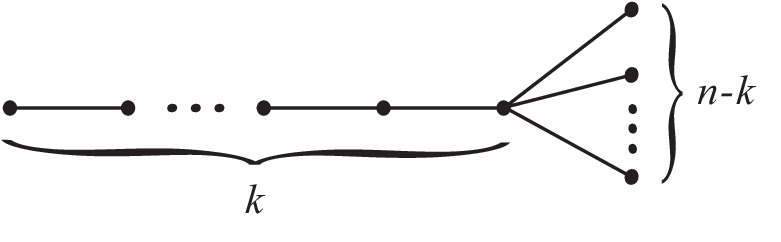}
\caption{\label{fig1}\small The Broom graph  $B(n,k)$. }
\end{center}
\end{figure}
\begin{theorem}\label{thm1}
Let $T$ be a tree with $n$ vertices and diameter at least $k$, Then
$$\pi(T)\geq\biggl[1+\sqrt{2}-\frac{\sqrt{2}}{2(n-k+1)}\biggl]\biggl(\frac{1+\sqrt{2}}{2}\biggl)^{k-2}
+\biggl[1-\sqrt{2}+\frac{\sqrt{2}}{2(n-k+1)}\biggl]\biggl(\frac{1-\sqrt{2}}{2}\biggl)^{k-2},$$
where equality holds if and only if tree $T$ is the broom graph $B(n,k)$.
\end{theorem}

This paper is organized as follows. In Section 2, we first develop  some properties of the permanent of Laplacian matrices of graphs, and we introduce some graph operations that can be considered as graph transformations which will have smaller Laplacian ratios than that of the original graphs. Based on results as above, we give the proof of Theorem \ref{thm1}.  A brief summary will be  given in section 3.

\section{Main results}

Before proceeding to the proof of Theorem \ref{thm1}, we first prove some lemmas.

Lemma \ref{sss} follows from the definition of permanent and is well-known.
\begin{lemma}\label{sss}
Let $A$ be any $n\times n$ matrix and $i,j$ be indices  between $1$ and $n$. Then
$${\rm per} A=\sum_{i}a_{ij}{\rm per} A_{ij} ~{\rm and}~{\rm per}A=\sum_{j}a_{ij}{\rm per}A_{ij},$$
where $A_{ij}$ is the $(n-1)\times(n-1)$ matrix obtained by striking out the $i$th row and  $j$th column of $A$.
\end{lemma}
Let $G$ be a simple graph, and let $S\subseteq V(G)$ be a vertex subset. We use $L_{S}(G)$ to denote the principal submatrix of $L(G)$ formed by deleting the rows and columns corresponding to
all vertices of $S\subseteq V(G)$. In particular, if $S=\{v\}$ with $v\in V(G)$, then
$L_{\{v\}}(G)$ is simply written as $L_{v}(G)$.

Suppose that $v$ is a vertex of $G$. Define $\mathscr {C}_{G}(v)$ to be the set of cycles containing the vertex $v$ in $G$, and $N(v)$ the neighborhood of $v$.
\begin{lemma}
Let $G$ be a simple graph, and $v\in V(G)$ be a vertex.  Then
\begin{eqnarray*}
{\rm per} L(G)=d(v)~{\rm per} L_{v}(G)+\sum\limits_{u\in N(v)}{\rm per }L_{vu}(G)+2\sum\limits_{C\in \mathscr {C}_{G}(v)}(-1)^{|V(C)|}{\rm per }L_{V(C)}(G).
\end{eqnarray*}
\end{lemma}
\begin{proof}
Suppose that $V(G)=\{v_{1}$, $v_{2}$, $\ldots$, $v_{n}\}$. Consider a term $a_{1\sigma(1)}a_{2\sigma(2)}\cdots a_{n\sigma(n)}$ in the expansion of ${\rm per}L(G)$, where $a_{ij}$ is the $(i,j)$-entry of $L(G)$.  Note that $a_{ii}=d(v_{i})$, and when $i\neq j$, if $v_{i}v_{j}\in E(G)$, then $a_{ij}=-1$;  otherwise $a_{ij}=0$,  Hence if $a_{1\sigma(1)}a_{2\sigma(2)}\cdots a_{n\sigma(n)}\neq 0$, then  $j=\sigma(j)$ or $v_{j\sigma_{j}}\in E(G)$. Next, we only consider the situation that the permanent expansion of $L(G)$ is not $0$.  Without loss of generality, let $v=v_{1}$. Hence $a_{11}=d(v)$. Let $\sigma$ be a set of disjoint cyclic  permutation.  When $\gamma_{1}$ is a cycle in $\sigma$, we write $\sigma=\gamma_{1}\sigma^{'}$. Define $P=\{\sigma|a_{1\sigma(1)}a_{2\sigma(2)}\cdots a_{n\sigma(n)}\neq0\}$. According to the length of $\gamma_{1}$, $\sigma$ has three situations as follows:

$S_{1}=\{\sigma\in P, \gamma_{1}=(1)$,  i.e., $\sigma$ fixes $1\}$,

$S_{2}^{j}=\{\sigma\in P, \gamma_{1}=(1j)$,  i.e., $\gamma_{1}$ corresponds to an edge $vv_{j}$ of $G\}$,

$S_{C}=\{\sigma\in P, \gamma_{1}$ corresponds to a cycle $C$ of length $l\geq3$ which contains the vertex $v$ of $G\}$.

Then, as $S_{1}, S_{2}^{j}$ (for any $j\neq1$) and $S_{C}$ are mutually disjoint and $P=S_{1}\cup S_{2}^{j}\cup S_{C}$, we have

\begin{eqnarray}
&&{\rm per} L(G)\nonumber\\
&=&\sum_{\sigma}a_{1\sigma(1)}a_{2\sigma(2)}\cdots a_{n\sigma(n)}\nonumber\\
&=&\sum_{\sigma\in S_{1}}a_{1\sigma(1)}a_{2\sigma(2)}\cdots a_{n\sigma(n)}+\sum_{j}\sum_{\sigma\in S_{2}^{j}}a_{1\sigma(1)}a_{2\sigma(2)}\cdots a_{n\sigma(n)}\nonumber\\
&&+\sum_{C}\sum_{\sigma\in S_{C}}a_{1\sigma(1)}a_{2\sigma(2)}\cdots a_{n\sigma(n)}\nonumber\\
&=&d(v)\sum_{\sigma^{'} \atop \gamma_{1}\sigma^{'}\in S_{1}}a_{2\sigma^{'}(2)}\cdots a_{n\sigma^{'}(n)}
+\sum_{vv_{j}\in E(G)}\sum_{\sigma^{'} \atop  \gamma_{1}\sigma^{'}\in S_{2}^{j}}a_{2\sigma^{'}(2)}\cdots a_{(j-1)\sigma^{'}(j-1)}a_{(j+1)\sigma^{'}(j+1)}\nonumber\\
&&\cdots a_{n\sigma^{'}(n)}+2\sum_{C\in \mathscr {C}_{G}(v)}(-1)^{|V(C)|}\sum_{\sigma^{'} \atop  \gamma_{1}\sigma^{'}\in S_{C}}a_{i_{1}\sigma^{'}(i_{1})}a_{i_{2}\sigma^{'}(i_{2})}\cdots a_{i_{n-l}\sigma^{'}(i_{i_{n-l}})}\nonumber\\
&=&d(v)~{\rm per} L_{v}(G)+\sum\limits_{vv_{j}\in E(G)}{\rm per }L_{vu}(G)+2\sum\limits_{C\in \mathscr {C}_{G}(v)}(-1)^{|V(C)|}{\rm per }L_{V(C)}(G),\nonumber
\end{eqnarray}
where in the penultimate step $i_{1}$, $\ldots$, $i_{n-l}$ are indices of the vertices of the graph  $G-V(C)$.
\end{proof}
\begin{lemma}
If $v$ is a pendant vertex of $G$,  and $u$ is the vertex adjacent to $v$. Then
\begin{eqnarray*}
{\rm per} L(G)={\rm per} L(G-v)+2~{\rm per} L_{u}(G-v).
\end{eqnarray*}
\end{lemma}

\begin{proof}
By Lemma 2.2 with $d_{v}=1$ and $\mathscr {C}_{G}=\phi$, we have ${\rm per} L(G)={\rm per} L_{v}(G)+{\rm per} L_{uv}(G)$. Using the linearity of the permanent as a function of any row or column, it can be known that
${\rm per} L_{v}(G)={\rm per} L(G-v)+{\rm per} L_{uv}(G)$.
By the definition of ${\rm per} L_{uv}(G)$, we have
${\rm per} L_{uv}(G)={\rm per} L_{u}(G-v)$.
According to arguments as above, we obtain that
$${\rm per} L(G)={\rm per} L(G-v)+2~{\rm per} L_{uv}(G)={\rm per} L(G-v)+2~{\rm per} L_{u}(G-v).$$
\end{proof}
\begin{lemma}
Let $v$ be a vertex of $G$. Suppose that $H$ is a proper subgraph of $G$ satisfying $v\notin V(H)$. Then
\begin{eqnarray*}
{\rm per} L_{V(H)}(G)&=&d(v)~{\rm per} L_ {V(H)\cup \{v\}}(G)+\sum\limits_{uv\in E(G) \atop u\notin V(H)}{\rm per }L_ {V(H)\cup \{u,v\}}(G)\\
&&+2\sum\limits_{C\in \mathscr {C}_{G}(v) \atop V(C)\cap V(H)=\emptyset}(-1)^{|V(C)|}{\rm per }L_{V(C)\cup V(H)}(G).
\end{eqnarray*}
\end{lemma}
\begin{proof}
The proof is similar to that of Lemma 2.2, Without loss of generality, we denote
$V(H)=\{v_{h+1}, v_{h+2}, ..., v_{n}\}$. Then, as $S_{1}, S_{2}^{j}$ (for any $j\neq1$) and $S_{C}$ are mutually disjoint and $P=S_{1}\cup S_{2}^{j}\cup S_{C}$, we have
\begin{eqnarray}\label{equ22}
&&{\rm per}~L_{V(H)}(G)\nonumber\\
&=&\sum_{\sigma \atop \gamma_{1}\sigma^{'} \in S_{1}}a_{1\sigma (1)}a_{2\sigma (2)}\cdots a_{h\sigma (h)}\nonumber\\
&=&\sum_{\sigma \in S_{1}}a_{1\sigma (1)}a_{2\sigma (2)}\cdots a_{h\sigma (h)}+\sum_{j}\sum_{\sigma \in S_{2}^{j}}a_{1\sigma (1)}a_{2\sigma (2)}\cdots a_{h\sigma (h)}\nonumber\\
&&+\sum_{C}\sum_{\sigma \in S_{C}}a_{1\sigma (1)}a_{2\sigma (2)}\cdots a_{h\sigma (h)}\nonumber\\
&=&d(v)\sum_{\sigma^{'} \atop \gamma_{1}\sigma^{'}\in S_{1}}a_{2\sigma^{'}(2)}\cdots a_{h\sigma^{'}(h)}
+\sum_{vv_{j}\in E(G)}\sum_{\sigma^{'} \atop  \gamma_{1}\sigma^{'}\in S_{2}^{j}}a_{2\sigma^{'}(2)}\cdots a_{(j-1)\sigma^{'}(j-1)}a_{(j+1)\sigma^{'}(j+1)}\nonumber\\
&&\cdots a_{h\sigma^{'}(h)}+2\sum_{C\in \mathscr {C}_{G}(v)}(-1)^{|V(C)|}\sum_{\sigma^{'} \atop  \gamma_{1}\sigma^{'}\in S_{C}}a_{i_{1}\sigma^{'}(i_{1})}a_{i_{2}\sigma^{'}(i_{2})}\cdots a_{i_{n-h-l}\sigma^{'}(i_{i_{n-h-l}})}\nonumber\\
&=&d(v)~{\rm per} L_ {V(H)\cup \{v\}}(G)+\sum\limits_{vv_{j}\in E(G) \atop u\notin V(H)}{\rm per }L_ {V(H)\cup \{u,v\}}(G)\nonumber\\
&&+2\sum\limits_{C\in \mathscr {C}_{G}(v) \atop V(C)\cap V(H)=\emptyset}(-1)^{|V(C)|}{\rm per }L_{V(C)\cup V(H)}(G),\nonumber
\end{eqnarray}
where in the penultimate step $i_{1}, \ldots, i_{n-h-l}$ are indices of the vertices of the graph $G-\{V(H)$\\
$\cup V(C)\}$.
\end{proof}
\begin{lemma}
Let $e=uv$ be an edge of $G$, and let $\mathscr {C}_{G}(e)$ be a set of cycles containing the edge $e$ in $G$. Then
\begin{eqnarray*}
{\rm per} L(G)&=&{\rm per} L(G-e)+{\rm per} L_{v}(G-e)+{\rm per} L_{u}(G-e)\\
&&+2~{\rm per} L_{uv}(G)+2\sum\limits_{C\in \mathscr {C}_{G}(e)}(-1)^{|V(C)|}{\rm per }L_{V(C)}(G).
\end{eqnarray*}
\end{lemma}

\begin{proof}
By Lemma 2.2, we have
\begin{eqnarray*}
{\rm per} L(G)=d(u)~{\rm per} L_{u}(G)+\sum\limits_{uw\in E(G)}{\rm per }L_{uw}(G)+2\sum\limits_{C\in \mathscr {C}_{G}(u)}(-1)^{|V(C)|}{\rm per }L_{V(C)}(G)
\end{eqnarray*}
and
\begin{eqnarray*}
{\rm per} L(G-e)&=&(d(u)-1)~{\rm per} L_{u}(G-e)+\sum\limits_{uw\in E(G-e)}{\rm per }L_{uw}(G-e)\\
&&+2\sum\limits_{C\in \mathscr {C}_{G-e}(u)}(-1)^{|V(C)|}{\rm per }L_{V(C)}(G-e).
\end{eqnarray*}
Therefore,
\begin{eqnarray}\label{equ23}
&&{\rm per} L(G)-{\rm per} L(G-e)\nonumber\\
&=&d(u)~{\rm per} L_{u}(G)-(d(u)-1)~{\rm per} L_{u}(G-e)+\biggl[\sum\limits_{uw\in E(G)}{\rm per }L_{uw}(G)\nonumber\\
&&-\sum\limits_{uw\in E(G-e)}{\rm per }L_{uw}(G-e)\biggl]+2~\biggl[\sum\limits_{C\in \mathscr {C}_{G}(u)}(-1)^{|V(C)|}{\rm per }L_{V(C)}(G)\\
&&-\sum\limits_{C\in \mathscr {C}_{G-e}(u)}(-1)^{|V(C)|}{\rm per }L_{V(C)}(G-e)\biggl].\nonumber
\end{eqnarray}
Since
\begin{eqnarray}\label{equ24}
{\rm per} L_{u}(G)={\rm per} L_{u}(G-e)+{\rm per }L_{uv}(G-e)
={\rm per} L_{u}(G-e)+{\rm per }L_{uv}(G).
\end{eqnarray}
Similarly, we have both
\begin{eqnarray}\label{equ25}
\sum_{uw\in E(G) \atop w\neq v}{\rm per }L_{uw}(G)=\sum_{uw\in E(G-e)}{\rm per }L_{uw}(G-e)+\sum_{uw\in E(G) \atop w\neq v}{\rm per }L_{\{u,w,v\}}(G)
\end{eqnarray}
and
\begin{eqnarray}\label{equ26}
\sum_{C\in \mathscr {C}_{G-e}(u) \atop v\notin V(C)}(-1)^{|V(C)|}{\rm per }L_{V(C)}(G)&=&\sum_{C\in \mathscr {C}_{G-e}(u) \atop v\notin V(C)}(-1)^{|V(C)|}{\rm per }L_{V(C)}(G-e)\nonumber\\
&&+\sum_{C\in \mathscr {C}_{G-e}(u) \atop v\notin V(C)}(-1)^{|V(C)|}{\rm per }L_{V(C)\cup \{v\}}(G-e).
\end{eqnarray}
Substituting (\ref{equ24}),  (\ref{equ25}), (\ref{equ26}) into (\ref{equ23}), we have
\begin{eqnarray}\label{equ27}
&&{\rm per} L(G)-{\rm per} L(G-e)\nonumber\\
&=&(d(u)+1)~{\rm per }L_{uv}(G)+{\rm per }L_{u}(G-e)\nonumber\\
&&+\sum_{uw\in E(G) \atop w\neq v}{\rm per }L_{\{u,w,v\}}(G)+2\sum\limits_{C\in \mathscr {C}_{G}(e)}(-1)^{|V(C)|}{\rm per }L_{V(C)}(G)\nonumber\\
&&+2\sum_{C\in \mathscr {C}_{G-e}(u) \atop v\notin V(C)}(-1)^{|V(C)|}{\rm per }L_{V(C)\cup \{v\}}(G-e).
\end{eqnarray}
By Lemma 2.4, we have
\begin{eqnarray}\label{equ28}
{\rm per} L_{v}(G-e)&=&(d(u)-1)~{\rm per }L_{uv}(G-e)+\sum_{uw\in E(G-e) \atop w\neq v}{\rm per }L_{\{u,w,v\}}(G-e)\nonumber\\
&&+2\sum_{C\in \mathscr {C}_{G-e}(u) \atop v\notin V(C)}(-1)^{|V(C)|}{\rm per }L_{V(C)\cup \{v\}}(G-e).
\end{eqnarray}
Substituting (\ref{equ28}) into (\ref{equ27}), we have
\begin{eqnarray*}
{\rm per} L(G)-{\rm per} L(G-e)&=&{\rm per} L_{v}(G-e)+{\rm per} L_{u}(G-e)\\
&&+2~{\rm per} L_{uv}(G)+2\sum\limits_{C\in \mathscr {C}_{G}(e)}(-1)^{|V(C)|}{\rm per }L_{V(C)}(G).
\end{eqnarray*}
It follow that
\begin{eqnarray*}
{\rm per} L(G)&=&{\rm per} L(G-e)+{\rm per} L_{v}(G-e)+{\rm per} L_{u}(G-e)\\
&&+2~{\rm per} L_{uv}(G)+2\sum\limits_{C\in \mathscr {C}_{G}(e)}(-1)^{|V(C)|}{\rm per }L_{V(C)}(G).
\end{eqnarray*}
\end{proof}
Let $k$ be a nonnegative  integer, and let $\mu$ be a $k$-matching of $G$ (i.e., a set of $k$ vertex disjoint edges) which meets vertices $v_{_{i_{1}}}, v_{_{i_{2}}}, \ldots,  v_{_{i_{2k}}}$ such that $v_{_{i_{j}}}$ has degree $d_{_{i_{j}}}$, $1\leq j\leq2k$. Define
$$d_{G}(\mu)=
\begin{cases}
d_{_{i_{1}}}\cdots d_{_{i_{2k}}}~~~{\rm if}~ k>0\\
~~~~~1~~~~~~~~~~{\rm if}~k=0
\end{cases}$$

In particular, if $e=u_{1}u_{2}\in E(G)$ and the vertices  $u_{1}, u_{2}$ has degrees $d_{_{u_{1}}}$ and $d_{_{u_{2}}}$, then
$$d_{G}(e)=d_{_{u_{1}}} d_{_{u_{2}}}.$$
When the graph $G$ is understood from the context, we also omit the subscript $G$ and simply write  $d(\mu)$ and $d(e)$ for $d_{G}(\mu)$ and $d_{G}(e)$, respectively.
Let $\mathscr{G}_{k}(G)$ denote the set of all $k$-matchings of $G$. We define $\pi_{k}(G)$ by
$$\pi_{k}(G)=\sum\limits_{\mu\in\mathscr{G}_{k}(G)}\frac{1}{d(\mu)}.$$
\begin{lemma}(\cite{bru})
Let $T$ be a tree, then
\begin{eqnarray*}
\pi(T)=\sum\limits_{k=0}\limits^{\lfloor\frac{n}{2} \rfloor}\pi_{k}(T)
\end{eqnarray*}
\end{lemma}
\begin{lemma}(\cite{bru})
Let $T$ be a tree with diameter $k$. Then $\pi(T)\geq\pi(P_{k+1})$ with equality holding  if and only if $T$ is the path $P_{k+1}$, or when $k=2$, $T$ is the star $S_{n}$.
\end{lemma}

For an integer $r\geq1$, let $T_{r}$ denote a tree of order $r(r\geq1)$. Suppose that $H$ is a tree disjoint from $T_{r}$
Assume that $u\in V(H)$ and $v\in V(T_{r})$. Let  $G_{1}$ be the graph obtained from the union of $H$ and $T_{r}$ by  linking $u$ and $v$ with a new edge, and $G_{2}$ be the graph obtained from $H$  by attaching $r$ pendent edges to $u$.
Suppose in $G_{2}$, the edges incident  with $u$  in $H$ are labelled $h_{1}, h_{2}, \cdots, h_{s}$, and other $r$ pendent edges incident with
 $u$  in $G_{2}$ are labelled $e_{1}, e_{2}, \cdots, e_{r}$.  Suppose in $G_{1}$, the edges in $T_{r}$ are labelled $e_{2}^{'}, e_{3}^{'}, \cdots, e_{r}^{'}$, and  $e_{1}^{'}=uv$. See Figure 1 for illustrations of $G_{1}$ and $G_{2}$.   By the definitions of $G_{1}$ and $G_{2}$, we have  $d_{G_{2}}(u)=d_{G_{1}}(u)+r-1$.
Let $\mathscr{G}_{k}(G_{2})$  denote the set of all $k$-matchings of $G_{2}$, let $\mathscr{G}_{k,u}^{+}(G_{2})$  be the set of all $k$-matchings  covering $u$, and  $\mathscr{G}^{-}_{k,u}(G_{2})$ be the set of all $k$-matchings not covering $u$. By these definitions, $\mathscr{G}_{k,u}^{+}(G_{2})\cup\mathscr{G}^{-}_{k,u}(G_{2})=\mathscr{G}_{k}(G_{2})$.
Denote by $\mathscr{G}_{k,u}^{++}(G_{2})$ the set of all $k$-matching in $\mathscr{G}_{k,u}^{+}(G_{2})$ which covers exactly one edge of $\{e_{1}, e_{2}, \cdots, e_{r}\}$, and denote by $\mathscr{G}_{k,u}^{+-}(G_{2})$ the set of all k-matching in $\mathscr{G}_{k,u}^{+}(G_{2})$ which covers exactly one edge of $\{h_{1}, h_{2}, \cdots, h_{s}\}$.  By their definitions, $\mathscr{G}_{k,u}^{++}(G_{2})\cup\mathscr{G}_{k,u}^{+-}(G_{2})=\mathscr{G}_{k,u}^{+}(G_{2})$.
$\mathscr{G}_{k}(G_{1})$  denotes the set of all $k$-matchings of $G_{1}$.  Denote by $\mathscr{G}_{k}^{+}(G_{1})$  the set of all $k$-matchings in $\mathscr{G}_{k}(G_{1})$ which  covers exactly one edge of $\{h_{1}, h_{2}, \cdots, h_{s}, e_{1}^{'}, e_{2}^{'}, \cdots, e_{r}^{'}\}$. Let $\mathscr{G}^{-}_{k}(G_{1})$ be the set of all $k$-matchings uncovering $u$ in $H$, and $\mathscr{G}^{*}_{k}(G_{1})=\mathscr{G}_{k}(G_{1})-\mathscr{G}_{k}^{+}(G_{1})-\mathscr{G}^{-}_{k}(G_{1})$. Denote by $\mathscr{G}_{k}^{++}(G_{1})$ the set of all k-matching in $\mathscr{G}_{k}^{+}(G_{1})$ which covers one edge of $\{e_{1}^{'}, e_{2}^{'}, \cdots, e_{r}^{'}\}$, and denote by $\mathscr{G}_{k}^{+-}(G_{1})$ the set of all k-matching in $\mathscr{G}_{k}^{+}(G_{1})$ which covers one edge of $\{h_{1}, h_{2}, \cdots, h_{s}\}$. By their definitions, $\mathscr{G}_{k}^{++}(G_{2})\cup\mathscr{G}^{+-}_{k}(G_{2})=\mathscr{G}_{k}^{+}(G_{2})$.
\begin{figure}[htbp]
\begin{center}
\includegraphics[scale=0.5]{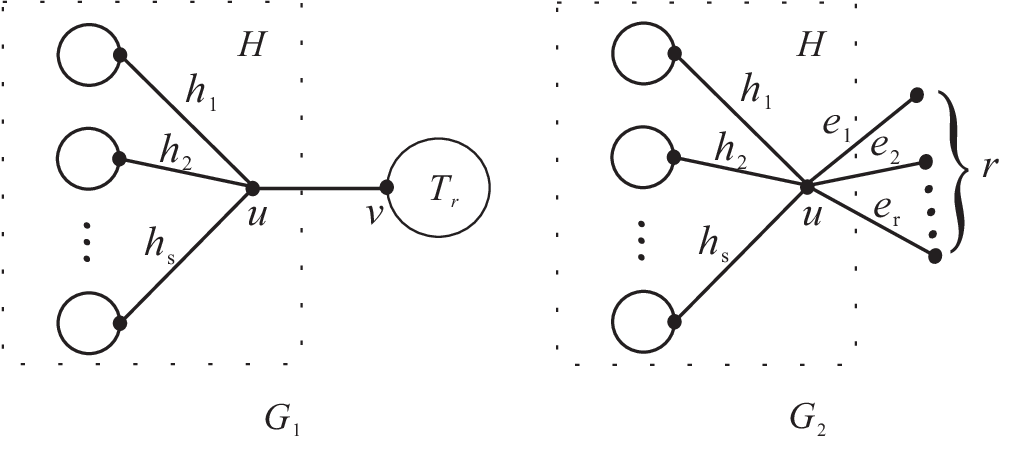}
\caption{\label{fig2}\small Graphs $G_{1}$ and $G_{2}$. }
\end{center}
\end{figure}

Define
\begin{eqnarray}\label{equ29}
\psi=\frac{1}{d(e_{1}^{'})}+\frac{1}{d(e_{2}^{'})}+\cdots+\frac{1}{d(e_{r}^{'})}~{\rm and }~ \theta=\frac{1}{d(e_{1})}+\frac{1}{d(e_{2})}+\cdots+\frac{1}{d(e_{r})}.
\end{eqnarray}
\begin{lemma}
Let $G_{1}$ and $G_{2}$ be two graphs defined as above. Then

(i) If $\psi\geq\theta$, then $\pi(G_{1})\geq \pi(G_{2})$.

(ii) If ~$T_{r}$ is a star and $v$ is the center  of $T_{r}$, then $\pi(G_{1})\geq\pi(G_{2})$.

(iii) $\pi(G_{1})\geq\pi(G_{2})$.
\end{lemma}

\begin{proof} \textbf{(i)}. If $r=1$ in $G_{1}$ and $G_{2}$, then $G_{1}\cong G_{2}$. Therefore $\pi(G_{1})=\pi(G_{2})$. Now, set $r\geq2$. By the definitions of $\pi_{k}(G_{1})$ and $\pi_{k}(G_{2})$, we have $\pi_{k}(G_{2})=\sum\limits_{\mu\in\mathscr{G}^{-}_{k,u}(G_{2})}\frac{1}{d(\mu)}+\sum\limits_{\mu\in\mathscr{G}^{+}_{k,u}(G_{2})}\frac{1}{d(\mu)}$ and $\pi_{k}(G_{1})\geq\sum\limits_{\mu\in\mathscr{G}^{-}_{k}(G_{1})}\frac{1}{d(\mu)}+\sum\limits_{\mu\in\mathscr{G}^{+}_{k}(G_{1})}\frac{1}{d(\mu)}$.
Checking the structure of a $k$-matching in $\mathscr{G}^{++}_{k,u}(G_{2})$, since every $k$-matching in $\mathscr{G}^{++}_{k,u}(G_{2})$ contains exactly one edge of $\{e_{1}, e_{2}, \cdots, e_{r}\}$ and a $(k-1)$-matching in $\mathscr{G}^{-}_{k,u}(G_{2})$,  we have $\sum\limits_{\mu\in\mathscr{G}^{++}_{k,u}(G_{2})}\frac{1}{d(\mu)}=\sum\limits_{\mu\in\mathscr{G}^{-}_{k-1,u}(G_{2})}\frac{1}{d(\mu)}\theta$. Similarly, by the definition of $G_{1}$ and the distributions of a $k$-matching in $\mathscr{G}^{++}_{k}(G_{1})$, we observe that every $k$-matching in $\mathscr{G}^{++}_{k}(G_{1})$ contains exactly one edge of $\{e_{1}^{'}, e_{2}^{'}, \cdots, e_{r}^{'}\}$ and a $(k-1)$-matching in $\mathscr{G}^{-}_{k}(G_{1})$. This yields $\sum\limits_{\mu\in\mathscr{G}^{++}_{k}(G_{1})}\frac{1}{d(\mu)}=\sum\limits_{\mu\in\mathscr{G}^{-}_{k-1}(G_{1})}\frac{1}{d(\mu)}\psi$. Hence,
\begin{eqnarray*}
\pi_{k}(G_{2})&=&\sum\limits_{\mu\in\mathscr{G}^{-}_{k,u}(G_{2})}\frac{1}{d(\mu)}+\sum\limits_{\mu\in\mathscr{G}^{-}_{k,u}(G_{2})}\frac{1}{d(\mu)}\\
&=&\sum\limits_{\mu\in\mathscr{G}^{-}_{k,u}(G_{2})}\frac{1}{d(\mu)}+\sum\limits_{\mu\in\mathscr{G}^{++}_{k,u}(G_{2})}\frac{1}{d(\mu)}+\sum\limits_{\mu\in\mathscr{G}^{+-}_{k,u}(G_{2})}\frac{1}{d(\mu)}\\
&=&\sum\limits_{\mu\in\mathscr{G}^{-}_{k,u}(G_{2})}\frac{1}{d(\mu)}+\sum\limits_{\mu\in\mathscr{G}^{-}_{k-1,u}(G_{2})}\frac{1}{d(\mu)}\theta+\sum\limits_{\mu\in\mathscr{G}^{++}_{k,u}(G_{2})}\frac{1}{d(\mu)},
\end{eqnarray*}
\begin{eqnarray*}
\pi_{k}(G_{1})&\geq&\sum\limits_{\mu\in\mathscr{G}^{-}_{k}(G_{1})}\frac{1}{d(\mu)}+\sum\limits_{\mu\in\mathscr{G}^{+}_{k}(G_{1})}\frac{1}{d(\mu)}\\
&=&\sum\limits_{\mu\in\mathscr{G}^{-}_{k}(G_{1})}\frac{1}{d(\mu)}+\sum\limits_{\mu\in\mathscr{G}^{++}_{k}(G_{1})}\frac{1}{d(\mu)}+\sum\limits_{\mu\in\mathscr{G}^{+-}_{k}(G_{1})}\frac{1}{d(\mu)}\\
&=&\sum\limits_{\mu\in\mathscr{G}^{-}_{k}(G_{1})}\frac{1}{d(\mu)}+\sum\limits_{\mu\in\mathscr{G}^{-}_{k-1}(G_{1})}\frac{1}{d(\mu)}\psi+\sum\limits_{\mu\in\mathscr{G}^{+-}_{k}(G_{1})}\frac{1}{d(\mu)}.
\end{eqnarray*}
Checking the structure of  $G_{1}$ and $G_{2}$, we know that $\mathscr{G}^{-}_{k,u}(G_{2})=\mathscr{G}^{-}_{k}(G_{1})$. Therefore
\begin{eqnarray*}\label{equ30}
\sum\limits_{\mu\in\mathscr{G}^{-}_{k,u}(G_{2})}\frac{1}{d(\mu)}=\sum\limits_{\mu\in\mathscr{G}^{-}_{k}(G_{1})}\frac{1}{d(\mu)}~{\rm and }~ \sum\limits_{\mu\in\mathscr{G}^{-}_{k-1,u}(G_{2})}\frac{1}{d(\mu)}=\sum\limits_{\mu\in\mathscr{G}^{-}_{k-1}(G_{1})}\frac{1}{d(\mu)}.
\end{eqnarray*}

Since $\psi\geq\theta$, we obtain
$\sum\limits_{\mu\in\mathscr{G}^{-}_{k}(G_{1})}\frac{1}{d(\mu)}\psi\geq\sum\limits_{\mu\in\mathscr{G}^{-}_{k-1,u}(G_{2})}\frac{1}{d(\mu)}\theta$.

By the definitions of $G_{1}$ and $G_{2}$, since $r>1$, we have $d_{G_{1}}(u)<d_{G_{2}}(u)$. Therefore, we also have $d_{G_{1}}(h_{i})<d_{G_{2}}(h_{i})(1\leq i\leq s)$. This implies that  $\sum\limits_{\mu\in\mathscr{G}^{+-}_{k}(G_{1})}\frac{1}{d(\mu)}>\sum\limits_{\mu\in\mathscr{G}^{+-}_{k,u}(G_{2})}\frac{1}{d(\mu)}$.

Combining the arguments above,  we conduce that $\pi_{k}(G_{1})\geq\pi_{k}(G_{2})$. By Lemma 2.6, we obtain  $\pi(G_{1})\geq\pi(G_{2})$.

\textbf{(ii)}. If $r=1$ in $G_{1}$ and $G_{2}$, then $G_{1}\cong G_{2}$. Therefore $\pi(G_{1})=\pi(G_{2})$. Now, set $r\geq2$.
Checking the structures of $G_{1}$ and $G_{2}$, set $T_{r}$ is a star, and $v$ is the center  of $T_{r}$ in $G_{1}$.  By (\ref{equ29}), we have $\psi=\frac{1}{d_{u}r}+\frac{r-1}{r}~{\rm and }~ \theta=\frac{r}{d_{u}+r-1}$. By the definitions  of $G_{1}$ and $G_{2}$, if $d_{u}=1$ in $G_{1}$, then $G_{1}\cong G_{2}$. This implies that $\pi(G_{1})=\pi(G_{2})$. If $d_{u}=2$ in $G_{1}$, then $\psi-\theta=\frac{1}{2r}+\frac{r-1}{r}-\frac{r}{2+r-1}
>\frac{r-1}{2r(r+1)}\geq0$. If $d_{u}\geq 3$ in $G_{1}$, then $\psi-\theta>\frac{r-1}{r}-\frac{r}{3+r-1}
>\frac{r-2}{r(r+2)}\geq0$. By (i) of Lemma 2.8, we have $\pi(G_{1})\geq\pi(G_{2})$.

\textbf{(iii)}. If $r=1$ in  $G_{1}$. By the definitions of $G_{1}$ and $G_{2}$, it can be known that $G_{1}\cong G_{2}$. Therefore $\pi(G_{1})=\pi(G_{2})$. Now, set $r\geq2$.  We  consider two cases as follows.

\textbf{Case 1.} Suppose that the induced subgraph with vertices  $V(T_{r})\cup\{u\}$ in $G_{1}$  is   isomorphic to  $P_{r}$. By (\ref{equ29}), we have $\psi=\frac{1}{2d_{u}}+\frac{r-2}{4}+\frac{1}{2}~{\rm and }~\theta=\frac{r}{d_{u}+r-1}$. Since $r\geq 2$ and $d_{u}\geq1$, $\psi-\theta=\frac{1}{2d_{u}}+\frac{r}{4}-\frac{r}{d_{u}+r-1}=\frac{d_{u}(r^{2}+d_{u}r+2-5r)+2r-2}{4d_{u}(d_{u}+r-1)}
\geq\frac{r(r-2)}{4d_{u}(d_{u}+r-1)}\geq0.$
By (i) of Lemma 2.8, we have $\pi(G_{1})\geq\pi(G_{2})$.

\textbf{Case 2.} Suppose that the induced subgraph of the vertices of $V(T_{r})\cup\{u\}$ in $G_{1}$ is  not isomorphic to $P_{r}$.  The structure of $G_{1}$ in this case is displayed on Figure 2. There exists the longest $v, w$-path  in $T_{r}$ and the vertices of the $v,w$-path are labeling $v(w_{0})$, $w_{1}, w_{2}, \cdots, w_{m}$, $w$ and their degrees in $G_{1}$ is $d_{v}, d_{w_{1}}, d_{w_{2}}, \cdots, d_{w_{m}}, d_{w}$, respectively.  By the definition of $G_{1}$ as illustrated in Figure 2, we observe that $G_{1}$ contains a  vertex $w_{i}$ with $i\geq0$ and $d_{w_{i}}\geq3$, and with $d(v,w_{i})$, the length of a shortest $v, w$-path in $G$, being minimized
Let $V(T_{r-i+1})= V(T_{r})-\{w_{1}, w_{2}, \cdots, w_{i-1}\}$,  and $H^{'}$ is $G[V(H)\cup\{v, w_{1}, w_{2}, \cdots, w_{i-1}\}]$. Let $G_{1}^{'}$ be a tree obtained from $H^{'}$ by attaching $r-i$ pendent vertices to $w_{i-1}$. See Figure 2 for the structures of $H^{'}$. Direct computation yields that $PD(G_{1})=\frac{PD(H^{'})}{d_{w_{i-1}}-1}\times d_{w_{i-1}}\times\frac{d_{w_{i}}PD(T_{r-i})}{d_{w_{i}}-1}$. By the definitions of $G_{1}$ and $G_{1}^{'}$, it is easy to obtain that $d_{G_{1}^{'}}(w_{i-1})=d_{w_{i-1}}+r-i$.
\begin{figure}[htbp]
\begin{center}
\includegraphics[scale=0.7]{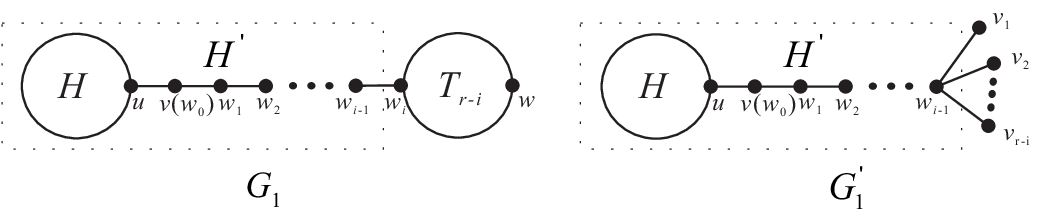}
\caption{\label{fig3}\small Graphs $G_{1}$ and $G_{1}^{'}$. }
\end{center}
\end{figure}
With these  definitions and notion, we will prove $\pi(G_{1})>\pi(G_{1}^{'})$ and $\pi(G_{1}^{'})\geq\pi(G_{2})$. To prove $\pi(G_{1})>\pi(G_{1}^{'})$, we consider different values of the diameter of $T_{r}$.

If the diameter of $T_{r-i}$ is $2$. Then $G_{1}\cong G_{1}^{'}$,  $\pi(G_{1})=\pi(G_{1}^{'})$.

Assume that the diameter of  $T_{r-i}$ is $3$, then $T_{r-i}$ is a double-star. Assume $T_{r-i}=DS(p,q)$ and $p+q=r-i$. Suppose that the vertices $c_{1}, c_{2}$  is two centers of $DS(p,q)$ and the degree sequence of  $DS(p,q)$ is $p,q,1,\cdots,1$. Since $d_{w_{i}}\geq3$, $w_{i}$ is $c_{1}$ or $c_{2}$. Without loss of generality, set $w_{i}$ is $c_{1}$ (or $c_{2}$). $G_{1}^{*}$  denotes the tree obtained by deleting pendent vertices of $c_{2}$ (or $c_{1}$) in $G_{1}$ and attaching simultaneously these pendent vertices to $c_{1}$ (or $c_{2}$). By (ii) of Lemma 2.8, we have $\pi(G_{1})>\pi(G_{1}^{*})$. By (ii) of Lemma 2.8 again, we have $\pi(G_{1}^{*})>\pi(G_{2})$. Therefore $\pi(G_{1})>\pi(G_{1}^{'})$.

Suppose that  the diameter of  $T_{r-i}$ is at least $4$. By Lemma 2.5, we have
\begin{eqnarray*}
{\rm per}~L(G_{1})
&=&{\rm per} L(G_{1}-w_{i-1}w_{i})+{\rm per} L_{w_{i-1}}(G_{1}-w_{i-1}w_{i})\\
&&+{\rm per} L_{w_{i}}(G_{1}-w_{i-1}w_{i})+{\rm per} L_{w_{i-1}w_{i}}(G_{1})\\
&=&{\rm per} L(H^{'})~{\rm per} L(T_{r-i})+{\rm per} L_{w_{i-1}}(H^{'})~{\rm per} L(T_{r-i})\\
&&+{\rm per} L(H^{'})~{\rm per} L_{w_{i}}(T_{r-i})+2~{\rm per} L_{w_{i-1}}(H^{'})~{\rm per} L_{w_{i}}(T_{r-i})\\
&=&[{\rm per}~L(T_{r-i})+{\rm per}~L_{w_{i}}(T_{r-i})]{\rm per}~L(H^{'})\\
&&+[{\rm per}~L(T_{r-i})+2~{\rm per}~L_{w_{i}}(T_{r-i})]{\rm per}~L_{w_{i-1}}(H^{'}).
\end{eqnarray*}
By the definitions  of $G_{1}^{'}$ and $H^{'}$, we have ${\rm per}~L_{v_{1}w_{i-1}}(G_{1}^{'})={\rm per}~L_{w_{i-1}}(H^{'})$.  By Lemma 2.3, and by deleting $v_{1}, v_{2}, \cdots, v_{r-i}$ one by one in $G_{1}^{'}$, we obtain that
\begin{eqnarray*}
&&{\rm per}~L(G_{1}^{'})\\
&=&{\rm per}~L(G_{1}^{'}-v_{1})+2~{\rm per}~L_{v_{1}w_{i-1}}(G_{1}^{'})\\
&=&{\rm per}~L(G_{1}^{'}-v_{1})+2~{\rm per}~L_{w_{i-1}}(H^{'})\\
&=&{\rm per}~L(G_{1}^{'}-v_{1}-v_{2})+2~{\rm per}~L_{w_{i-1}}(H^{'})+2~{\rm per}~L_{w_{i-1}}(H^{'})\\
&=&{\rm per}~L(G_{1}^{'}-v_{1}-v_{2}-v_{3})+2~{\rm per}~L_{w_{i-1}}(H^{'})+2~{\rm per}~L_{w_{i-1}}(H^{'})+2~{\rm per}~L_{w_{i-1}}(H^{'})\\
&&\cdots\\
&=&{\rm per}~L(G_{1}^{'}-v_{1}-v_{2}-\cdots-v_{r-i-1}-v_{r-i})+2(r-i)~{\rm per}~L_{w_{i-1}}(H^{'})\\
&=&{\rm per}~L(H^{'})+2(r-i)~{\rm per}~L_{w_{i-1}}(H^{'}).
\end{eqnarray*}
By Definition of Laplacian ratio, we obtain that
\begin{eqnarray}\label{equ315}
&&\pi(G_{1})-\pi(G_{1}^{'})\nonumber\\
&=&\frac{[{\rm per}~L(T_{r-i})+{\rm per}~L_{w_{i}}(T_{r-i})]{\rm per}~L(H^{'})+[{\rm per}~L(T_{r-i})}{\frac{PD(H^{'})}{d_{w_{i-1}}-1}\times d_{w_{i-1}} \times \frac{d_{w_{i}}}{d_{w_{i}}-1}PD(T_{r-i})}\nonumber\\
&&+\frac{2~{\rm per}~L_{w_{i}}(T_{r-i})]{\rm per}~L_{w_{i-1}}(H^{'})}{\frac{PD(H^{'})}{d_{w_{i-1}}-1}\times d_{w_{i-1}} \times \frac{d_{w_{i}}}{d_{w_{i}}-1}PD(T_{r-i})}-\frac{{\rm per}~L(H^{'})+2(r-i)~{\rm per}~L_{w_{i-1}}(H^{'})}{\frac{PD(H^{'})}{d_{w_{i-1}}-1}\times (d_{w_{i-1}}+r-i-1) \times 1^{r-i}}\nonumber\\
&=&\frac{(d_{w_{i-1}}+r-i-1)[{\rm per}~L(T_{r-i-1})+{\rm per}~L_{w_{i}}(T_{r-i})]{\rm per}~L(H^{'})}{\frac{PD(H^{'})}{d_{w_{i-1}}-1}\times d_{w_{i-1}} \times \frac{d_{w_{i}}}{d_{w_{i}}-1}PD(T_{r-i})\times (d_{w_{i-1}}+r-i-1)}\nonumber\\
&&+\frac{(d_{w_{i-1}}+r-i-1)[{\rm per}~L(T_{r-i})+2~{\rm per}~L_{w_{i}}(T_{r-i})]{\rm per}~L_{w_{i-1}}(H^{'})}{\frac{PD(H^{'})}{d_{w_{i-1}}-1}\times d_{w_{i-1}} \times \frac{d_{w_{i}}}{d_{w_{i}}-1}PD(T_{r-i})\times (d_{w_{i-1}}+r-i-1)}\nonumber\\
&&-\frac{d_{w_{i-1}}\times\frac{d_{w_{i}}}{d_{w_{i}}-1}PD(T_{r-i})[{\rm per}~L(H^{'})+2(r-i)~{\rm per}~L_{w_{i-1}}(H^{'})]}{\frac{PD(H^{'})}{d_{w_{i-1}}-1}\times d_{w_{i-1}} \times \frac{d_{w_{i}}}{d_{w_{i}}-1}PD(T_{r-i})\times (d_{w_{i-1}}+r-i-1)}\nonumber\\
&=&\frac{[(d_{w_{i-1}}+r-i-1)~{\rm per}~L(T_{r-i})+(d_{w_{i-1}}+r-i-1)~{\rm per}~L_{w_{i}}(T_{r-i})]{\rm per}~L(H^{'})}{\frac{PD(H^{'})}{d_{w_{i-1}}-1}\times d_{w_{i-1}} \times \frac{d_{w_{i}}}{d_{w_{i}}-1}PD(T_{r-i})\times (d_{w_{i-1}}+r-i-1)}\nonumber\\
&&-\frac{d_{w_{i-1}}\times\frac{d_{w_{i}}}{d_{w_{i}}-1}PD(T_{r-i}){\rm per}~L(H^{'})}{\frac{PD(H^{'})}{d_{w_{i-1}}-1}\times d_{w_{i-1}} \times\frac{d_{w_{i}}}{d_{w_{i}}-1}PD(T_{r-i})\times (d_{w_{i-1}}+r-i-1)}\\
&&+\frac{[(d_{w_{i-1}}+r-i-1)~{\rm per}~L(T_{r-i})+2~(d_{w_{i-1}}+r-i-1)~{\rm per}~L_{w_{i}}(T_{r-i})]{\rm per}~L_{w_{i-1}}(H^{'})}{\frac{PD(H^{'})}{d_{w_{i-1}}-1}\times d_{w_{i-1}} \times \frac{d_{w_{i}}}{d_{w_{i}}-1}PD(T_{r-i})\times (d_{w_{i-1}}+r-i-1)}\nonumber\\
&&-\frac{2(r-i)d_{w_{i-1}}\times\frac{d_{w_{i}}}{d_{w_{i}}-1}PD(T_{r-i}){\rm per}~L(H^{'})}{\frac{PD(H^{'})}{d_{w_{i-1}}-1}\times d_{w_{i-1}} \times \frac{d_{w_{i}}}{d_{w_{i}}-1}PD(T_{r-i})\times (d_{w_{i-1}}+r-1)}.\nonumber
\end{eqnarray}
Since $r\geq2$, $\frac{PD(H^{'})}{d_{w_{i-1}}-1}\times d_{w_{i-1}} \times \frac{d_{w_{i}}}{d_{w_{i}}-1}PD(T_{r-i})\times (d_{w_{i-1}}+r-i-1)>0$. By the definition of $T_{r-i}$ and utilizing  the structure of the matrix $L_{w_{i}}(T_{r-i})$, we obtain ${\rm per}~L_{w_{i}}(T_{r-i})\geq1$. By Lemma 2.1, we have ${\rm per}~L(H^{'})>(d_{w_{i-1}}-1){\rm per}~L_{w_{i-1}}(H^{'})$.
Since $T_{r-i}$ have diameter at least $4$ and $\pi(P_{5})=3$, by Lemma 2.7, we have $\pi(T_{r-i})\geq3$.  Since $d_{w_{i}}\geq3$, we know that ${\rm per}~L(T_{r-i})\geq3~PD(T_{r-i})\geq2~\frac{d_{w_{i}}}{d_{w_{i}}-1}PD(T_{r-i})$. Let $\Gamma$ denote the numerator of $\pi(G_{1})-\pi(G_{1}^{'})$, see (\ref{equ315}).
 Hence
\begin{eqnarray*}
\Gamma&=&[(d_{w_{i-1}}+r-i-1)~{\rm per}~L(T_{r-i})+(d_{w_{i-1}}+r-i-1)~{\rm per}~L_{w_{i}}(T_{r-i})\\
&&-d_{w_{i-1}} \times\frac{d_{w_{i}}}{d_{w_{i}}-1}PD(T_{r-i})]{\rm per}~L(H^{'})+[(d_{w_{i-1}}+r-i-1)~{\rm per}~L(T_{r-i})\\
&&+2~(d_{w_{i-1}}+r-i-1)~{\rm per}~L_{w_{i}}(T_{r-i})]{\rm per}~L_{w_{i-1}}(H^{'})\\
&&-2(r-i)d_{w_{i-1}} \times\frac{d_{w_{i}}}{d_{w_{i}}-1}PD(T_{r-i}){\rm per}~L_{w_{i-1}}(H^{'})\\
&>&[2(d_{w_{i-1}}+r-i-1) \times\frac{d_{w_{i}}}{d_{w_{i}}-1}PD(T_{r-i})+(d_{w_{i-1}}+r-i-1)\\
&&-d_{w_{i}-1} \times\frac{d_{w_{i}}}{d_{w_{i}}-1}PD(T_{r-i})](d_{w_{i-1}}-1){\rm per}~L_{w_{i-1}}(H^{'})\\
&&+[2(d_{w_{i-1}}+r-i-1) \times\frac{d_{w_{i}}}{d_{w_{i}}-1}PD(T_{r-i})+2~(d_{w_{i-1}}+r-i-1)\\
&&-2rd_{w_{i}-1} \times\frac{d_{w_{i}}}{d_{w_{i}}-1}PD(T_{r-i})]{\rm per}~L_{w_{i-1}}(H^{'})\\
&=&(d_{w_{i-1}}^{2}-d_{w_{i-1}}) \times\frac{d_{w_{i}}}{d_{w_{i}}-1}PD(T_{r-i})+(d_{w_{i-1}}+1)(d_{w_{i-1}}+r-i-1)\\
&&+(r-i) \times\frac{d_{w_{i}}}{d_{w_{i}}-1}PD(T_{r-i}).
\end{eqnarray*}
Since  $d_{w_{i-1}}\geq2$, $r>i$, we have $\pi(G_{1})-\pi(G_{1}^{'})>0$.

Combining arguments  above, it is easy to obtain that $\pi(G_{1})>\pi(G_{1}^{'})$.

We are now ready to prove $\pi(G_{1}^{'})\geq\pi(G_{2})$.

By the definitions of $G_{1}^{'}$,  by (\ref{equ29}), we have $\psi=\frac{1}{d_{u}d_{v}}+\frac{1}{d_{w_{0}}d_{w_{1}}}+\frac{1}{d_{w_{1}}d_{w_{2}}}+\cdots+\frac{1}{d_{w_{i-2}}d_{w_{i-1}}}
+\frac{1}{d_{w_{i-1}}d_{v_{1}}}+\frac{1}{d_{w_{i-1}}d_{v_{2}}}+\cdots+\frac{1}{d_{w_{i-1}}d_{v_{r-i}}}$ and $\theta=\frac{r}{d_{u}+r-1}$.

Assume $i\geq3$. We obtain $\psi=\frac{1}{2d_{u}}+\frac{i-2}{4}+\frac{1}{2(r-i+1)}+\frac{r-i}{r-i+1}$.
Since $d_{w_{1}}=2$, $d_{w_{i-2}}=2$ and $d_{u}\geq1$, it is easy to obtain  that $\psi-\theta>\frac{i-2}{4}-\frac{1}{2(r-i+1)}\geq0$.

Assume that $i=2$. Then $r\geq i+1\geq3$. If $d_{u}=1$, then $\psi-\theta=\frac{1}{2}[1-\frac{1}{(r-1)}]>0$. If $d_{u}\geq2$, then $\psi-\theta>\frac{r-3}{2(r+1)(r-1)}\geq0$.

Assume $i=1$. By (ii) of Lemma 2.8, we have $\pi(G_{1}^{'})-\pi(G_{2})>0$.

Set $i=0$. We obtain that $G_{1}^{'}\cong G_{2}$. Hence $\pi(G_{1}^{'})=\pi(G_{2})$.

Combining arguments as above,  we obtain that $\pi(G_{1}^{'})\geq \pi(G_{2})$.
\end{proof}

\begin{lemma}
Let $s, t$ be positive integers satisfying $s\geq t$, $T^{''}$ be a tree with $n\geq4$ vertices which contains two vertices of degree $2$ $u$ and $v$. Let $T$ be a tree obtained from $T^{''}$ by attaching $s$ and $t$ pendant vertices to $u$ and $v$, respectively. Denote by $T_{1}$  the tree obtained from $T^{''}$ by attaching $s+t$ pendant vertices to $u$. Denote by $T_{2}$ the tree obtained from $T^{''}$ by attaching $s+t$ pendant vertices to $v$. See Figure 3 for illustrations of these graphs. Then $\pi(T)>{\rm min}\{\pi(T_{1}), \pi(T_{2})\}$.
\end{lemma}
\begin{figure}[htbp]
\begin{center}
\includegraphics[scale=0.5]{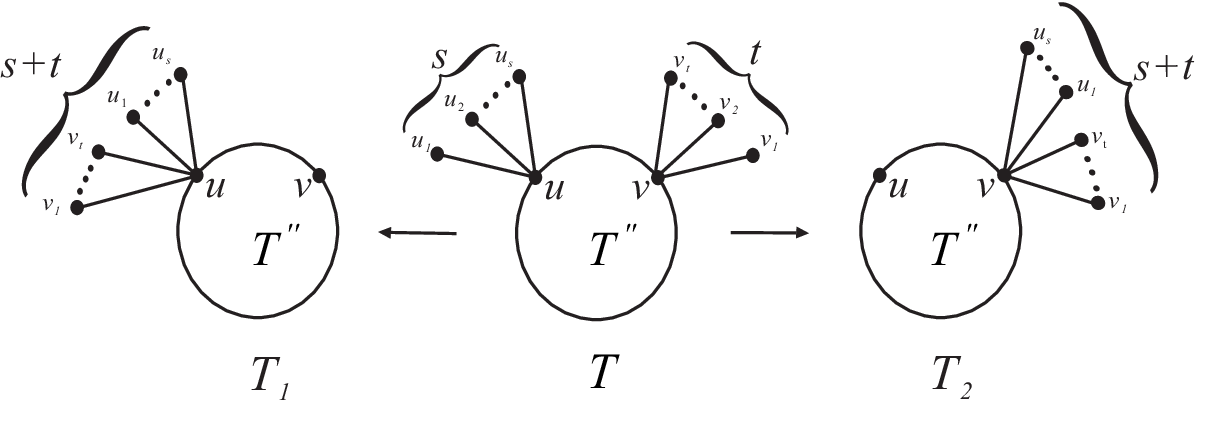}
\caption{\label{fig4}\small Graphs $T$, $T_{1}$ and $T_{2}$. }
\end{center}
\end{figure}
\begin{proof}
Suppose that $T^{'}$ is a tree obtained from $T$ by deleting all  vertices  $u_{1}, u_{2}, \cdots, u_{s}$. By the definitions of $T$, $T_{1}$, $T_{2}$, $T^{'}$ and $T^{''}$, we obtain that
\begin{eqnarray*}
{\rm per} L_{u}(T-u_{1})&=&{\rm per} L_{u}(T-u_{1}-u_{2})\\
&=&\cdots\\
&=&{\rm per} L_{u}(T-u_{1}-u_{2}-\cdots-u_{s-1}-u_{s})\\
&=&{\rm per} L_{u}(T^{'}),\\
{\rm per} L_{v}(T^{'}-v_{1})&=&{\rm per} L_{v}(T^{'}-v_{1}-v_{2})\\
&=&\cdots\\
&=&{\rm per} L_{v}(T^{'}-v_{1}-v_{2}-\cdots-v_{t-1}-v_{t})\\
&=&{\rm per} L_{v}(T^{''}),\\
{\rm per} L_{u}(T_{1}-u_{1})&=&{\rm per} L_{u}(T_{1}-u_{1}-u_{2})\\
&=&\cdots\\
&=&{\rm per} L_{u}(T_{1}-u_{1}-u_{2}-\cdots-u_{s-1}-u_{s})\\
&=&{\rm per} L_{u}(T_{1}-u_{1}-u_{2}-\cdots-u_{s-1}-u_{s}-v_{1})\\
&=&{\rm per} L_{u}(T_{1}-u_{1}-u_{2}-\cdots-u_{s-1}-u_{s}-v_{1}-v_{2})\\
&=&\cdots\\
&=&{\rm per} L_{u}(T_{1}-u_{1}-u_{2}-\cdots-u_{s-1}-u_{s}-v_{1}-v_{2}-\cdots-v_{t-1}-v_{t})\\
&=&{\rm per} L_{u}(T^{''})
\end{eqnarray*}
and
\begin{eqnarray*}
{\rm per} L_{v}(T_{2}-u_{1})&=&{\rm per} L_{v}(T_{2}-u_{1}-u_{2})\\
&=&\cdots\\
&=&{\rm per} L_{v}(T_{2}-u_{1}-u_{2}-\cdots-u_{s-1}-u_{s})\\
&=&{\rm per} L_{v}(T_{2}-u_{1}-u_{2}-\cdots-u_{s-1}-u_{s}-v_{1})\\
&=&{\rm per} L_{v}(T_{2}-u_{1}-u_{2}-\cdots-u_{s-1}-u_{s}-v_{1}-v_{2})\\
&=&\cdots\\
&=&{\rm per} L_{v}(T_{2}-u_{1}-u_{2}-\cdots-u_{s-1}-u_{s}-v_{1}-v_{2}-\cdots-v_{t-1}-v_{t})\\
&=&{\rm per} L_{v}(T^{''}).
\end{eqnarray*}

By Lemma 2.3, deleting $u_{1}$, $u_{2}$, $\cdots$, $u_{s}$, $v_{1}$, $v_{2}$, $\cdots$, $v_{t}$ one by one in $T$, $T_{1}$ and $T_{2}$, we obtain that
\begin{eqnarray*}
{\rm per} L(T)&=&{\rm per} L(T-u_{1})+2~{\rm per} L_{u}(T-u_{1})\\
&=&{\rm per} L(T-u_{1}-u_{2})+2~{\rm per} L_{u}(T-u_{1})+2~{\rm per} L_{u}(T-u_{1}-u_{2})\\
&=&\cdots\\
&=&{\rm per} L(T-u_{1}-u_{2}-\cdots-u_{s})+2~{\rm per} L_{u}(T-u_{1})+2~{\rm per} L_{u}(T-u_{1}-u_{2})\\
&&+\cdots+2~{\rm per} L_{u}(T-u_{1}-u_{2}-\cdots-u_{s-1}-u_{s})\\
&=&{\rm per} L(T^{'})+2s~{\rm per} L_{u}(T^{'})\\
&=&{\rm per} L(T^{'}-v_{1})+2s~{\rm per} L_{u}(T^{'})+2~{\rm per} L_{v}(T^{'}-v_{1})\\
&=&{\rm per} L(T^{'}-v_{1}-v_{2})+2s~{\rm per} L_{u}(T^{'})+2~{\rm per} L_{v}(T^{'}-v_{1})+2~{\rm per} L_{v}(T^{'}-v_{1}-v_{2})\\
&=&\cdots\\
&=&{\rm per} L(T^{'}-v_{1}-v_{2}-\cdots-v_{t-1}-v_{t})+2s~{\rm per} L_{u}(T^{'})+2~{\rm per} L_{v}(T^{'}-v_{1})\\
&&+2~{\rm per} L_{v}(T^{'}-v_{1}-v_{2})+\cdots+2{\rm per} L_{v}(T^{'}-v_{1}-v_{2}-\cdots-v_{t-1}-v_{t})\\
&=&{\rm per} L(T^{'}-v_{1}-v_{2}-\cdots-v_{t-1}-v_{t})+2s~{\rm per} L_{u}(T^{'})+2t~{\rm per} L_{v}(T^{''})\\
&=&{\rm per} L(T^{''})+2s~{\rm per} L_{u}(T^{'})+2t~{\rm per} L_{v}(T^{''}),
\end{eqnarray*}
\begin{eqnarray*}
{\rm per} L(T_{1})&=&{\rm per} L(T_{1}-u_{1})+2~{\rm per} L_{u}(T_{1}-u_{1})\\
&=&{\rm per} L(T_{1}-u_{1}-u_{2})+2~{\rm per} L_{u}(T_{1}-u_{1})+2~{\rm per} L_{u}(T_{1}-u_{1}-u_{2})\\
&=&\cdots\\
&=&{\rm per} L(T_{1}-u_{1}-u_{2}-\cdots-u_{s}-v_{1}-v_{2}-\cdots-v_{t-1}-v_{t})\\
&&+2~{\rm per} L_{u}(T_{1}-u_{1})+2~{\rm per} L_{u}(T_{1}-u_{1}-u_{2})+\cdots+2{\rm per} L_{u}(T_{1}\\
&&-u_{1}-u_{2}-\cdots-u_{s-1}-u_{s}-v_{1}-v_{2}-\cdots-v_{t-1}-v_{t})\\
&=&{\rm per} L(T^{''})+2(s+t)~{\rm per} L_{u}(T^{''}),
\end{eqnarray*}
\begin{eqnarray*}
{\rm per} L(T_{2})&=&{\rm per} L(T_{2}-u_{1})+2~{\rm per} L_{v}(T_{2}-u_{1})\\
&=&{\rm per} L(T_{2}-u_{1}-u_{2})+2~{\rm per} L_{v}(T_{2}-u_{1})+2~{\rm per} L_{v}(T_{2}-u_{1}-u_{2})\\
&=&\cdots\\
&=&{\rm per} L(T_{2}-u_{1}-u_{2}-\cdots-u_{s-1}-u_{s}-v_{1}-v_{2}-\cdots-v_{t-1}-v_{t})\\
&&+2~{\rm per} L_{v}(T_{2}-u_{1})+2~{\rm per} L_{v}(T_{2}-u_{1}-u_{2})+\cdots+2~{\rm per} L_{v}(T_{2}\\
&&-u_{1}-u_{2}-\cdots-u_{s-1}-u_{s}-v_{1}-v_{2}-\cdots-v_{t-1}-v_{t})\\
&=&{\rm per} L(T^{''})+2(s+t)~{\rm per} L_{v}(T^{''}).
\end{eqnarray*}

By the definitions of ~$T^{''}$, $T$, $T_{1}$ and $T_{2}$. Direct computing yields
 \begin{eqnarray*}
PD(T)&=&\frac{(s+2)(t+2)}{4}PD(T^{''}),\\
PD(T_{1})&=&\frac{(s+t+2)}{2}PD(T^{''}), \\
PD(T_{2})&=&\frac{(s+t+2)}{2}PD(T^{''}).
\end{eqnarray*}
 Hence
\begin{eqnarray*}
\pi(T)&=&\frac{{\rm per} L(T)}{\frac{(s+2)(t+2)}{4}PD(T^{''})}=\frac{4}{PD(T^{''})}[\frac{{\rm per} L(T^{''})}{(s+2)(t+2)}+\frac{2s~{\rm per} L_{u}(T^{'})+2t~{\rm per} L_{v}(T^{''})}{(s+2)(t+2)}],\\
\pi(T_{1})&=&\frac{{\rm per} L(T_{1})}{\frac{(s+t+2)}{2}PD(T^{''})}=\frac{4}{PD(T^{''})}[\frac{{\rm per} L(T^{''})}{2(s+t+2)}+\frac{2(s+t)~{\rm per} L_{u}(T^{''})}{2(s+t+2)}],\\
\pi(T_{2})&=&\frac{{\rm per} L(T_{2})}{\frac{(s+t+2)}{2}PD(T^{''})}=\frac{4}{PD(T^{''})}[\frac{{\rm per} L(T^{''})}{2(s+t+2)}+\frac{2(s+t)~{\rm per} L_{v}(T^{''})}{2(s+t+2)}].
\end{eqnarray*}

Since $d(v)=2$, the matrix $L(T^{''})$ has the following form

\begin{eqnarray}\label{equ35}
L(T^{''})=\left[
\begin{matrix}
\begin{array}{ccccc|c|ccccc}
  &      &  &  &    &    0  &           &    &       &\\
  &      &  &  &    &\vdots &           &    &       &\\
  &      & A & &     &    0  &           &   &       &\\
  &      &  &  &     &  0    &           &    &       &\\
  &      &  &  &     &  -1  &           &    &       &\\
\hline
0&\cdots &0 & 0& -1  &  2  &    -1  &  0 &  0 &\cdots & 0 \\
\hline
  &      &  &  &     &  -1  &       &   &    &       &\\
   &       & &  &     &   0  &      &     &  &       &\\
  &       & &  &     &   0  &       &   &   B &       &\\
  &       & &  &     &\vdots &      &    &    &        &\\
  &       & &  &     &   0  &       &   &    &       &  \\
\end{array}
\end{matrix}\right],
\end{eqnarray}
where $A$ and $B$ are two submatrices.

And the matrix $L_{u}(T^{''})$ has the following form

~$L_{u}(T^{''})=\left[
\begin{matrix}
\begin{array}{ccccc|c|ccccc|c}
  &      &  &  &    &    0  &           &    &       &  &\\
  &      &  &  &    &\vdots &           &    &       &  &\\
  &      & A & &     &    0  &           &   &       &   &\\
  &      &  &  &     &  0    &           &    &      &   &\\
  &      &  &  &     &  -1  &           &    &       &   &\\
\hline
0&\cdots &0 & 0& -1  &  2  &    -1  &  0 &  0 &\cdots & 0 &\\
\hline
  &      &  &  &     &  -1  &       &   &    &       &   & \\
   &       & &  &     &   0  &      &     &  &       &    &\\
  &       & &  &     &   0  &       &   &   B^{'} &   &    &\\
  &       & &  &     &\vdots &      &    &    &      &     &\\
  &       & &  &     &   0  &       &   &    &       &     &  \\
\hline
 &       &  &  &     &       &      &    &     &    &      & C\\
\end{array}
\end{matrix}\right]$,\\
where $A$ , $B^{'}$ and $C$ are three submatrices.

Similarly, since the vertex $v$ in $T^{'}$ has degree $t+2$, the matrix $L_{u}(T^{'})$ has the following form

~$L_{u}(T^{'})=\left[
\begin{matrix}
\begin{array}{ccccc|ccccc|ccccc|c}
  &      &  &  &    &    0  &     &    &         &     &       &    &    &       &   &\\
  &      &  &  &    &\vdots &     &    &         &     &      &     &    &       &   &\\
  &      & A & &     &    0  &    &    &         &     &      &     &    &       &   &\\
  &      &  &  &     &  0    &    &    &         &     &      &     &    &       &   &\\
  &      &  &  &     &  -1  &     &    &         &     &      &     &    &       &   &\\
\hline
0&\cdots &0 & 0& -1  &  t+2 &    -1     & -1 &  \cdots &  -1&   -1  &  0 &  0 &\cdots & 0 &\\
 &      &  &  &     &  -1  &    1      &    &         &     &      &     &   &       &   &\\
 &      &  &  &     &  -1  &           &  1 &         &     &      &     &   &       &   &\\
 &      &  &  &     &\vdots &          &    &\ddots   &     &      &     &   &       &   &\\
 &      &  &  &     &  -1  &           &    &         &  1  &      &     &   &       &   &\\
\hline
  &      &  &  &     &  -1  &          &     &        &      &     &     &    &       &   & \\
   &       & &  &     &   0  &         &     &        &     &      &     &    &       &   &\\
  &       & &  &     &   0  &          &     &        &     &      &     &   B^{'} &  &   &\\
  &       & &  &     &\vdots &         &     &        &     &      &     &    &       &   &\\
  &       & &  &     &   0  &          &     &        &     &      &     &    &       &   &  \\
\hline
 &       &  &  &     &       &         &     &        &     &      &     &     &       &   & C\\
\end{array}
\end{matrix}\right]$.

Suppose that $M$ is the matrix obtained by deleting the last row and the last column in $A$. Let $N$ be the matrix obtained by deleting the first row and the first column in $B$, and let $N^{'}$ be the matrix obtained by deleting the first row and the first column in $B^{'}$.  Direct computing yields
$${\rm per}L(T^{''})=2~{\rm per}A{\rm per}B+{\rm per}M{\rm per}B+{\rm per}A{\rm per}N,$$
\begin{eqnarray}\label{equ10}
{\rm per} L_{u}(T^{'})={\rm per}C[(s+s+2)~{\rm per}A{\rm per}B^{'}+{\rm per}M{\rm per}B^{'}+{\rm per}A{\rm per}N^{'}],
\end{eqnarray}
\begin{eqnarray}\label{equ11}
{\rm per} L_{u}(T^{''})={\rm per}C(2~{\rm per}A{\rm per}B^{'}+{\rm per}M{\rm per}B^{'}+{\rm per}A{\rm per}N^{'}).
\end{eqnarray}
Thus,
$$2~{\rm per}A{\rm per}B\leq{\rm per}L(T^{''})\leq4~{\rm per}A{\rm per}B,$$
\begin{eqnarray}\label{equ12}
2~{\rm per}A{\rm per}B^{'}\leq{\rm per} L_{u}(T^{''})\leq4~{\rm per}A{\rm per}B^{'},
\end{eqnarray}
Combining (\ref{equ10}), (\ref{equ11}) and (\ref{equ12}), we have
\begin{eqnarray}\label{equ13}
{\rm per} L_{u}(T^{'})\geq\frac{5(s+t)}{4}~{\rm per} L_{u}(T^{''}).
\end{eqnarray}
By (\ref{equ35}), we have ${\rm per}A{\rm per}B={\rm per} L_{u}(T^{''})$. Therefore $2{\rm per} L_{v}(T^{''})\leq{\rm per}L(T^{''})\leq4{\rm per} L_{v}(T^{''})$. Similarly  we have $2{\rm per} L_{u}(T^{''})\leq{\rm per}L(T^{''})\leq4{\rm per} L_{u}(T^{''})$.

\begin{eqnarray*}
\Delta_{1}&=&\pi(T)-\pi(T_{1})\\
&=&\frac{4}{PD(T^{''})}[\frac{{\rm per} L(T^{''})}{(s+2)(t+2)}+\frac{2s~{\rm per} L_{u}(T^{'})+2t~{\rm per}L_{v}(T^{''})}{(s+2)(t+2)}]\\
&&-\frac{4}{PD(T^{''})}[\frac{{\rm per} L(T^{''})}{2(s+t+2)}+\frac{2(s+t){\rm per} L_{u}(T^{''})}{2(s+t+2)}]\\
&=&\frac{4}{PD(T^{''})}[(\frac{1}{(s+2)(t+2)}-\frac{1}{2(s+t+2)})~{\rm per}L(T^{''})+(\frac{2s~{\rm per} L_{u}(T^{'})+2t~{\rm per} L_{v}(T^{''})}{(s+2)(t+2)}\\
&&-\frac{2(s+t)~{\rm per} L_{u}(T^{''})}{2(s+t+2)})]\\
&\geq&\frac{4}{PD(T^{''})}[(\frac{1}{(s+2)(t+2)}-\frac{1}{2(s+t+2)})\times4~{\rm per} L_{u}(T^{''})\\
&&+(\frac{2s~{\rm per} L_{u}(T^{'})+2t~{\rm per} L_{v}(T^{''})}{(s+2)(t+2)}-\frac{2(s+t)~{\rm per} L_{u}(T^{''})}{2(s+t+2)})].
\end{eqnarray*}

\begin{eqnarray*}
\Delta_{2}&=&\pi(T)-\pi(T_{2})\\
&=&\frac{4}{PD(T^{''})}[\frac{{\rm per} L(T^{''})}{(s+2)(t+2)}+\frac{2s~{\rm per} L_{u}(T^{'})+2t~{\rm per} L_{v}(T^{''})}{(s+2)(t+2)}]\\
&&-\frac{4}{PD(T^{''})}[\frac{{\rm per} L(T^{''})}{2(s+t+2)}+\frac{2(s+t)~{\rm per} L_{v}(T^{''})}{2(s+t+2)}]\\
&=&\frac{4}{PD(T^{''})}[(\frac{1}{(s+2)(t+2)}-\frac{1}{2(s+t+2)})~{\rm per}L(T^{''})+(\frac{2s~{\rm per} L_{u}(T^{'})+2t~{\rm per} L_{v}(T^{''})}{(s+2)(t+2)}\\
&&-\frac{2(s+t)~{\rm per} L_{v}(T^{''})}{2(s+t+2)})]\\
&\geq&\frac{4}{PD(T^{''})}[(\frac{1}{(s+2)(t+2)}-\frac{1}{2(s+t+2)})\times4~{\rm per} L_{u}(T^{''})\\
&&+(\frac{2s~{\rm per}L_{u}(T^{'})+2t~{\rm per} L_{v}(T^{''})}{(s+2)(t+2)}-\frac{2(s+t)~{\rm per} L_{v}(T^{''})}{2(s+t+2)})].
\end{eqnarray*}

If ${\rm per} L_{u}(T^{''})\geq{\rm per} L_{v}(T^{''})$, then we obtain from (\ref{equ13}) that
$$\Delta_{1}\geq\frac{(\frac{3}{4}t^{2}+2st+\frac{1}{2}t+\frac{3}{2}s+\frac{5}{4}s^{2}-2)s~{\rm per} L_{v}(T^{''})}{\frac{1}{2}(s+2)(t+2)(s+t+2)}.$$
Since $s\geq t\geq1$, it follows that $\Delta_{1}>0$.

If ${\rm per} L_{u}(T^{''})<{\rm per} L_{v}(T^{''})$, then we obtain from (\ref{equ13}) that
$$\Delta_{2}\geq\frac{(\frac{3}{4}st^{2}+2s^{2}t+\frac{1}{2}st+\frac{5}{4}s^{3}+\frac{3}{2}s^{2}+2t-2s)~{\rm per} L_{u}(T^{''})}{\frac{1}{2}(s+2)(t+2)(s+t+2)}.$$
Since $s\geq t\geq1$, it follows that $\Delta_{2}>0$.

Combining the arguments as above, we have $\pi(T)>{\rm min}\{\pi(T_{1}), \pi(T_{2})\}$.
\end{proof}

\begin{lemma}(\cite{bru})
Let $T$ be a tree with $n$ vertices and diameter at least $k$. Then
$${\rm per}L(T)\geq{\rm per}L(B(n,k)),$$
where equality holds if and only if $T$ is broom $B(n,k)$.
\end{lemma}

\begin{lemma}(\cite{bru})
Let $n\geq3$ and $k\geq2$. Suppose that $Q_{k}$ is the matrix obtained from $L(P_{n+1})$ by eliminating row $1$ and column $1$. Then
$${\rm per}L(B(n,k)=(2n-2k+1){\rm per}Q_{k-1}+{\rm per}Q_{k-2},$$
$${\rm per}Q_{k}=\frac{1}{2}(1+\sqrt{2})^{k}+\frac{1}{2}(1-\sqrt{2})^{k}.$$

\end{lemma}

\noindent\textbf{Proof of Theorem 1.1}. Let $T$ be a tree with $n$ vertices having diameter $k$.
Suppose  that  $v_{1}v_{2}\cdots v_{k+1}$ is a path with diameter $k$ in $T$. All branches  hanged on the vertex $v_{i}(i=2,3,\cdots,k)$ in  path $v_{1}v_{2}\cdots v_{k+1}$ are changed into a star with the center $v_{i}$, and the number of pendents vertices of the star equals to the number of vertices of all branches.
Then,  $T$ is transformed into a caterpillar $T_{0}$ whose spine is the path $v_{1}v_{2}\cdots v_{k+1}$. By (iii) of Lemma 2.8, we have $\pi(T)\geq\pi(T_{0})$.  Repeated merging the leaves of  vertices $v_{i}$ and $v_{j}$  into $v_{i}$ or $v_{j}$ in path $v_{1}v_{2}\cdots v_{k+1}$, $T_{0}$ is transformed into $T_{1}$ which is a new caterpillar obtained by joining $n-k-1$ vertices to the vertex $v_{s}$ of path $v_{1}v_{2}\cdots v_{k+1}$, where $2\leq i,j,s\leq k+1$. By  Lemma 2.9, it can be known that $\pi(T_{0})\geq\pi(T_{1})$. Since $PD(T_{1})=PD(B(n,k))$,  by Lemma 2.10, we have $\pi (T_{1})\geq\pi (B(n,k))$. For all broom graphs  with $n$ vertices having diameter at least $k$, by (ii) of  Lemma 2.8, it is easy to obtain that $\pi(B(n,k))$ is minimum. So,
$$\pi(T)\geq\frac{{\rm per}B(n,k)}{PD(B(n,k))},$$

By Lemma 2.11, we obtain that
$$\pi(T)\geq\biggl[1+\sqrt{2}-\frac{\sqrt{2}}{2(n-k+1)}\biggl]\biggl(\frac{1+\sqrt{2}}{2}\biggl)^{k-2}
+\biggl[1-\sqrt{2}+\frac{\sqrt{2}}{2(n-k+1)}\biggl]\biggl(\frac{1-\sqrt{2}}{2}\biggl)^{k-2}$$
with equality holding if and only if ~$T$ is broom $B(n,k)$. \qquad \qquad\qquad\qquad\qquad\qquad\qquad\quad$\Box$

\section{Summarize}

In this paper, we determined the lower bound of Laplacian ratios of trees, and the corresponding extremal graph is also determined. In general it is difficult to evaluate Laplacian ratios of graphs, unless the graphs are of very special type. It is precisely because of the difficulty of the calculation of Laplacian ratio that there are a few published papers on Laplacian ratio. Except for the problems raised by Brualdi and Goldwasser, there exist some problems to worth study. Such as: characterizing the bound of Laplacian ratios of general graphs, and determining the bound of Laplacian ratios of general graphs with given graph parameters, etc.. These problems are the direction of our future research.

\noindent{\bf Conflict of Interest Statement}

The authors declare that they have no conflicts of interest.


\begin{thebibliography}{99}

\bibitem{val}
L.G. Valiant, The complexity of computing the permanent, \textit{Theor. Comput. Sci.} 8 (1979) 189--201.

\bibitem{bru}
R.A. Brualdi, J.L. Goldwasser, Permanent of the Laplacian matrix of trees and bipartite graphs, \textit{ Discrete Math.} 48 (1984) 1--21.

\bibitem{gol}
J.L. Goldwasser, Permanent of the Laplacian matrix of trees with a given mathching, \textit{ Discrete Math.} 61 (1986) 197--212.

\end{thebibliography}
\end{document}